\documentclass[12pt, reqno]{amsart}
\usepackage{amsmath, amsthm, amscd, amsfonts, amssymb, graphicx, color}
\usepackage[bookmarksnumbered, colorlinks, plainpages]{hyperref}
\hypersetup{colorlinks=true,linkcolor=red, anchorcolor=green, citecolor=cyan, urlcolor=red, filecolor=magenta, pdftoolbar=true}

\newtheorem{theorem}{Theorem}[section]
\newtheorem{lemma}[theorem]{Lemma}
\newtheorem{proposition}[theorem]{Proposition}
\newtheorem{corollary}[theorem]{Corollary}

\newtheorem*{MTheorem}{Main Theorem}

\theoremstyle{definition}
\newtheorem{definition}[theorem]{Definition}
\newtheorem{example}[theorem]{Example}

\theoremstyle{remark}

\numberwithin{equation}{section}

\newcommand\C{\mathbb C}
\newcommand\D{\mathbb D}
\newcommand\R{\mathbb R}
\newcommand\T{\mathbb T}
\newcommand{\quotes}[1]{``#1''}

\newcommand\rtkmu{R^t(K, \mu)}

\newcommand\rikmu{R^\infty(K,\mu)}

\newcommand\limu{L^\infty(\mu)}
\renewcommand\i{\infty}
\newcommand\area{{\frak m}}
\newcommand\CT{{\mathcal C}}

\begin{document}

\setcounter{page}{1}

\title[Mean Rational Approximation]{Mean Rational Approximation for Some Compact Planar Subsets}

\author[J. Conway \MakeLowercase{and} L. Yang]{John B. Conway$^1$ \MakeLowercase{and} Liming Yang$^2$}

\address{$^{1}$Department of Mathematics, The George Washington University, Washington, DC 20052}
\email{\textcolor[rgb]{0.00,0.00,0.84}{conway@gwu.edu}}

\address{$^2$Department of Mathematics, Virginia Polytechnic and State University, Blacksburg, VA 24061.}
\email{\textcolor[rgb]{0.00,0.00,0.84}{yliming@vt.edu}}




\begin{abstract} In 1991, J. Thomson \cite{thomson} obtained celebrated structural results for $P^t(\mu).$ Later, J. Brennan \cite{b08} generalized Thomson's theorem to $\rtkmu$ when the diameters of the components of $\mathbb C\setminus K$ are bounded below. The results indicate  that if $R^t(K, \mu)$ is pure, then $\rtkmu \cap L^\infty (\mu)$ is the \quotes{same as} the algebra of bounded analytic functions on $\mbox{abpe}(R^t(K, \mu)),$ the set of analytic bounded point evaluations. We show that if the diameters of the components  of $\mathbb C\setminus K$ are allowed to tend to zero, then even though $\text{int}(K) = \mbox{abpe}(R^t(K, \mu))$ and $K =\overline {\text{int}(K)},$  the algebra $\rtkmu \cap L^\infty (\mu)$ may \quotes{be equal to} a proper sub-algebra of bounded analytic functions on $\text{int}(K),$ where functions in the sub-algebra are \quotes{continuous} on certain portions of the inner boundary of $K.$
\end{abstract} \maketitle

\section{\textbf{Introduction}}

For a Borel subset $E$ of the complex plane $\C,$ let $M_0(E)$ denote the set of finite complex-valued Borel measures that are compactly supported in $E$ and let $M_0^+(E)$ be the set of positive measures in $M_0(E).$ The support of $\nu\in M_0(\C),$ $\text{spt}(\nu),$ is the smallest closed set that has full $|\nu|$ measure. 
For a compact subset $K\subset \C,$ $\mu\in M_0^+(K),$ and $1\leq t \le \infty$, the analytic polynomials and functions in 
\[
\ \mbox{Rat}(K) := \{q:\mbox{$q$ is a rational function 
with poles off $K$}\}
\]
are members of $L^t(\mu)$. If $1\leq t < \infty,$ we let $P^t(\mu)$ denote the closure of the (analytic) 
polynomials in $L^t(\mu)$ and let $R^t(K, \mu)$ denote the closure of $\mbox{Rat}(K)$ in $L^t(\mu)$. Define
\[
\ \rtkmu ^\perp = \left \{g \in L^s(\mu):~ \int f(z)g(z)d\mu = 0, \text{ for } f \in \rtkmu \right \},
\]
where $\frac 1t + \frac 1s = 1.$
Let $\rikmu$ be the weak-star closure of $\mbox{Rat}(K)$ in $\limu.$ Denote 
 \[
 \ R^{t,\i}(K, \mu) = R^t(K, \mu) \cap L^\infty(\mu).
 \] 
We let
   $\partial_e K$ (the exterior boundary of $K$) denote the union of
   the boundaries of all the   
   components of $\C \setminus K.$ The inner boundary of $K$ is denoted $\partial_i K := \partial K \setminus \partial_e K.$
  
Let $\mathbb{D}$ denote the unit
disk $\{z: |z| < 1\}$, let $\T$ denote the unit circle $\{z: |z| = 1\}$, and let $\R$ denote the real line $\{z: \text{Im}(z) = 0\}.$ For $z_0\in \C$ and $r> 0,$ let $\D(z_0, r) := \{z:~|z - z_0| < r\}.$
Let $l_B$ be the arc-length measure on a rectifiable curve $\Gamma$ restricted to a subset $B \subset \Gamma.$ We use $C_1,C_2,...$ and $c_1,c_2,...$ for absolute constants that may change from one step to the next.  

A point $z_0$ in $\mathbb{C}$ is called a \textit{bounded point evaluation} for $P^t(\mu)$ (resp., $R^t(K, \mu)$)
if $f\mapsto f(z_0)$ defines a bounded linear functional, denoted $\phi_{z_0},$ for the analytic polynomials (resp., functions in $\mbox{Rat}(K)$)
with respect to the $L^t(\mu)$ norm. The collection of all such points is denoted $\mbox{bpe}(P^t(\mu))$ 
(resp., $\mbox{bpe}(R^t(K, \mu)$)).  If $z_0$ is in the interior of $\mbox{bpe}(P^t(\mu))$ (resp., $\mbox{bpe}(R^t(K, \mu)$)) 
and there exist positive constants $r$ and $M$ such that $|f(z)| \leq M\|f\|_{L^t(\mu)}$, whenever $z\in \D(z_0, r)$ 
and $f$ is an analytic polynomial (resp., $f\in \mbox{Rat}(K)$), then we say that $z_0$ is an 
\textit{analytic bounded point evaluation} for $P^t(\mu)$ (resp., $R^t(K, \mu)$). The collection of all such 
points is denoted $\mbox{abpe}(P^t(\mu))$ (resp., $\mbox{abpe}(R^t(K, \mu))$). Actually, it follows from Thomson's Theorem
\cite{thomson} (or see Theorem \ref{thmthom}, below) that $\mbox{abpe}(P^t(\mu))$ is the interior of $\mbox{bpe}(P^t(\mu))$. 
This also holds in the context of $R^t(K, \mu)$ as was shown by J. Conway and N. Elias in \cite{ce93}. Now, 
$\mbox{abpe}(P^t(\mu))$ is the largest open subset of $\mathbb{C}$ to which every function $f\in P^t(\mu)$ has an analytic 
continuation 
\[
\ \rho(f)(z) := \phi_z (f), ~ z\in \mbox{abpe}(P^t(\mu)),
\]
that is, 
\begin{eqnarray}\label{EFOnABPE}
\ f(z) = \rho(f)(z) ~ \mu|_{\mbox{abpe}(P^t(\mu))}-a.a.,
\end{eqnarray}
and similarly in the context of $R^t(K, \mu)$. The map $\rho$ is called the evaluation map.

Let $S_\mu$ denote the multiplication by $z$ on $P^t(\mu)$ (resp., $R^t(K, \mu)$).
The operator $S_\mu$ is pure if $P^t(\mu)$ (resp., $R^t(K, \mu)$) does not have a direct $L^t$ summand,
  and is irreducible if $P^t(\mu)$ (resp., $R^t(K, \mu)$) contains no non-trivial characteristic functions.
  For an open subset $U \subset \C,$ we let $H^\infty(U)$ denote the bounded analytic functions on $U$.
This is a Banach algebra in the maximum modulus norm, but it is important to endow it also with a weak-star topology.

We shall use $\overline {E}$ to denote the closure of the set $E\subset \C$ and $\text{int}(E)$ to denote the interior of $E.$ Our story begins with celebrated results of J. Thomson, in \cite{thomson}. 

\begin{theorem} (Thomson 1991)
\label{thmthom}
Let $\mu\in M_0^+(\C)$ and suppose that $1\leq t < \infty$.
Then there is a Borel partition $\{\Delta_i\}_{i=0}^\infty$ of $\mbox{spt}(\mu)$ such that 
\[
 \ P^t(\mu ) = L^t(\mu |_{\Delta_0})\oplus \bigoplus _{i = 1}^\infty P^t(\mu |_{\Delta_i})
 \]
and the following statements are true:

(a) If $i \ge 1$, then $S_{\mu |_{\Delta_i}}$ on $P^t(\mu |_{\Delta_i})$ is irreducible.

(b) If $i \ge 1$ and $U_i :=abpe( P^t(\mu |_{\Delta_i}))$, then $U_i$ is a simply connected region and $\Delta_i\subset \overline{U_i}$.

(c) If $S_\mu$ is pure (that is, if $\Delta_0 = \emptyset$) and $U = \text{abpe}(P^t(\mu))$, then the evaluation map $\rho$ is an isometric isomorphism and a weak-star homeomorphism from $P^t(\mu) \cap L^\infty(\mu)$ onto $H^\infty(U)$. 
\end{theorem}

The next result is due to Brennan \cite{b08}, which extends Theorem \ref{thmthom}.

\begin{theorem} (Brennan 2008)
\label{thmBCE}
 Let $\mu\in M_0^+(K)$ for a compact set $K$.
 Suppose that $1\le t < \infty$ and the diameters of the components of $\mathbb C\setminus K$ are bounded below. 
Then there is a Borel partition $\{\Delta_i\}_{i=0}^\infty$ of $\mbox{spt}(\mu)$ and compact subsets $\{K_i\}_{i=1}^\infty$ such that $\Delta_i \subset K_i$ for $i \ge 1$,
 \[
 \ R^t(K, \mu ) = L^t(\mu |_{\Delta_0})\oplus \bigoplus _{i = 1}^\infty R^t(K_i, \mu |_{\Delta_i}),
 \]
and the following statements are true:

(a) If $i \ge 1$, then $S_{\mu |_{\Delta_i}}$ on $R^t(K_i, \mu |_{\Delta_i})$ is irreducible.

(b) If  $i \ge 1$ and $U_i :=\mbox{abpe}( R^t(K_i, \mu |_{\Delta_i}))$, then $K_i = \overline{U_i}$.

(c) If  $i \ge 1$, then the evaluation map $\rho_i$ is an isometric isomorphism and a weak-star homeomorphism from $R^{t,\i}(K_i, \mu |_{\Delta_i})$ onto $H^\infty(U_i)$.
\end{theorem}
 
In this paper, we show that
    when the diameters of the components  of $\mathbb C\setminus K$ are allowed to tend to $0$, then
    Theorem \ref{thmBCE} (c) is no long valid, even though $\text{int}(K) = \text{abpe}(\rtkmu)$ and $K =\overline {\text{int}(K)}.$ In fact, the evaluation map is an isometric isomorphism and a weak-star homeomorphism from $R^{t,\i}(K, \mu)$ onto a proper sub-algebra of $H^\infty(\text{int}(K)),$ where functions in the sub-algebra have equal non-tangential limits on certain portions of $\partial_i K.$

\begin{definition}\label{StringOfBeadsDef}
Let $\{\lambda_n\}_{n = 1}^\infty$ 
be a subset of $\mathbb R$
such that $\D(\lambda _n, r_n)$ is a subset of $\mathbb D$ and $\overline{\D(\lambda _i, r_i)} \cap \overline{\D(\lambda _j, r_j)} = \emptyset$ for $i \ne j$. Define
 \[
 \ K \ :=\  \overline{\mathbb D} \setminus \left (\bigcup_{n = 1}^\infty \D(\lambda _n, r_n) \right ).
 \]
 Assume $ K \cap \mathbb R$ has no interior 
 as a subset of $\mathbb R$,
 and that the length of $K \cap  \mathbb R$ is greater than zero (i.e.
  $\sum _{n = 1}^\infty r_n < 1$).
  A set of this form is called a string of beads set. 
Set $\Omega = int(K)$. Define
$\Omega_U = \{z\in \Omega:~\text{Im}(z) > 0\}, ~ \Omega_L = \Omega\setminus \Omega_U$, $K_U = \overline{\Omega_U}$, and $K_L = \overline{\Omega_L}$.
\end{definition}

Clearly, $\partial_i K = K \cap \mathbb R, ~l_{\R}-a.a..$ For $\lambda\in \mathbb R$ and $0 < \alpha \le 1$, define the upper cone 
\begin{eqnarray}\label{ULConeDefEq1}
\ U_n(\lambda,\alpha) = \left \{z:~ |Re(z - \lambda )| < \alpha (Im(z - \lambda ))\right \}\cap \D \left (\lambda, \frac{1}{n} \right )
\end{eqnarray}
and the lower cone 
\begin{eqnarray}\label{ULConeDefEq2}
\ L_n(\lambda,\alpha) = \left \{z:~ |Re(z - \lambda )| < - \alpha (Im(z - \lambda ))\right \} \cap \D \left (\lambda ,\frac{1}{n} \right ).
\end{eqnarray}

It is easy to verify that for $\lambda \in \partial_i K, ~l_{\partial_i K} -a.a.,$ there exists $n_\lambda$ such that the non-tangential limit region
\begin{eqnarray}\label{ULConeDefEq3}
\ U_{n_\lambda}(\lambda,\alpha)\cup L_{n_\lambda}(\lambda,\alpha) \subset \Omega.
\end{eqnarray}

Fatou's theorem  says that bounded analytic functions in the unit disk have non-tangential limits almost everywhere. Precomposing with a Riemann map from $\D$ onto $\Omega_U$ we get that, for $f\in H^\infty (\Omega),$ the limits
\begin{eqnarray}\label{UNTLimitEq}
\ f_+(\lambda) := \lim_{z\in U_n(\lambda,\alpha)\rightarrow \lambda} f(z)
\end{eqnarray}
exist $l_{\partial_i K}-a.a.$ since the Riemann map is conformal from $\overline \D$ onto $\overline {\Omega_U}$ (see \cite[pp 208-218]{gm05} and \cite[pp 80-82]{pom92}), and likewise the limits
\begin{eqnarray}\label{LNTLimitEq}
\ f_-(\lambda) := \lim_{z\in L_n(\lambda,\alpha)\rightarrow \lambda} f(z)
\end{eqnarray}
exist $l_{\partial_i K}-a.a.$

For a subset $E_0\subset \partial_i K,$  $H^\i (\Omega, E_0)$ is the sub-algebra of $f\in H^\i (\Omega)$ for which $f_+(z) = f_-(z), ~ l_{E_0}-a.a..$ The set $E_0$ is called the removable boundary for $H^\i (\Omega, E_0).$
 
 The aim of this paper is to prove our main theorem below. By \eqref{EFOnABPE}, we will also use $f|_{\text{abpe}(R^t(K,\mu))}$ for $\rho(f).$

\begin{MTheorem}\label{SOBTheorem}
Let $1 \le t < \infty$ and let $K$ be a string of beads set defined as in Definition \ref{StringOfBeadsDef}. Suppose that $\mu\in M_0^+(K),$ $S_\mu$ on $R^t(K,\mu)$ is pure, and $\text{abpe}(R^t(K,\mu)) = \Omega .$ Then there exists a Borel subset $E_0\subset \partial_i K$ such that the following statements are true:

(1) If $f\in R^t(K, \mu)$, then the non-tangential limit from the top on $E_0$,
  $f_+(z)$,  and the non-tangential limit from 
  the bottom on $E_0$, $f_-(z)$, exist $l_{E_0}-a.a.$, and 
  \[
  \ f_+(z) = f_-(z),~ l_{E_0}-a.a..
  \] 
  Consequently, if $f\in R^{t,\i}(K, \mu),$ then $\rho(f) \in H^\i (\Omega, E_0).$ 

(2) If $F \in H^\i (\Omega, E_0),$ then there exists $f\in R^{t,\i}(K, \mu)$ such that $F = \rho(f).$

(3) The evaluation map $\rho$ is an isometric isomorphism and a weak-star homeomorphism from $R^{t,\i}(K, \mu)$ onto $H^\i (\Omega, E_0).$ 
\end{MTheorem}

Our methods can be applied to a compact subset $K\subset \C$ and $\mu\in M_0^+(K)$ if $S_\mu$ on $R^t(K,\mu)$ is pure, $\text{abpe}(R^t(K,\mu)) = \text{int}(K),$ $\partial_i K$ is a subset of a Lipschitz graph, and $\text{int}(K)$ contains non-tangential limit regions (similar to \eqref{ULConeDefEq3}) whose vertices are on $\partial_i K$ almost everywhere with respect to the arc-length measure restricted to $\partial_i K.$

In section 2, we review some recent results of analytic capacity and Cauchy transform that are needed in our analysis. In the context of $\rikmu,$ Proposition \ref{REEForSOB} is a simple version of Main Theorem (1) where $E_0$ is always equal to $\partial_i K.$ However, for $\rtkmu$ when $1\le t < \infty,$ $E_0$ can be an arbitrary subset of $\partial_i K$ as shown in Proposition \ref{SOBRemovable}. section 4 constructs the removable boundary $E_0$ and proves Main Theorem (1). Main Theorem (2) and (3) are proved in section 5 and section 6.

\section{\textbf{Prelimaries}}

Let $\nu \in M_0(\C).$
For $\epsilon > 0,$ $\mathcal C_\epsilon(\nu)$ is defined by
\ \begin{eqnarray}
\ \mathcal C_\epsilon(\nu)(z) = \int _{|w-z| > \epsilon}\dfrac{1}{w - z} d\nu (w).
\ \end{eqnarray} 
The (principal value) Cauchy transform
of $\nu$ is defined by
\ \begin{eqnarray}\label{CTDefinition}
\ \mathcal C(\nu)(z) = \lim_{\epsilon \rightarrow 0} \mathcal C_\epsilon(\nu)(z)
\ \end{eqnarray}
for all $z\in\mathbb{C}$ for which the limit exists.
If $\lambda \in \mathbb{C}$ and 
\begin{eqnarray}\label{CTAIDefinition}
\ \tilde \nu (\lambda) := \int \frac{d|\nu |}{|z - \lambda |} < \infty,	
\end{eqnarray}
 then 
$\lim_{r\rightarrow 0}\frac{|\nu |(\D (\lambda, r))}{r} = 0$ and 
$\lim_{\epsilon \rightarrow 0} \mathcal C_{\epsilon}(\nu)(\lambda )$ exists. Therefore, a standard application of Fubini's
Theorem shows that $\mathcal C(\nu) \in L^s_{\mbox{loc}}(\mathbb{C})$, for $ 0 < s < 2$. In particular, it is
defined for almost all $z$ with respect to the area measure $\area$ on $\mathbb{\mathbb{C}}$, and clearly $\mathcal{C}(\nu)$ is analytic 
in $\mathbb{C}_\infty \setminus\mbox{spt}(\nu)$, where $\mathbb{C}_\infty := \mathbb{C} \cup \{\infty \}.$ In fact, from 
Corollary \ref{ZeroAC} below, we see that \eqref{CTDefinition} is defined for all $z$ except for a set of zero analytic 
capacity. Throughout this paper, the Cauchy transform of a measure always means the principal value of the transform.

The maximal Cauchy transform is defined by
 \[
 \ \mathcal C_*(\nu)(z) = \sup _{\epsilon > 0}| \mathcal C_\epsilon(\nu)(z) |.
 \]

If $B \subset\C$ is a compact subset, then  we
define the analytic capacity of $B$ by
\[
\ \gamma(B) = \sup |f'(\infty)|,
\]
where the supremum is taken over all those functions $f$ that are analytic in $\mathbb C_{\infty} \setminus B$ such that
$|f(z)| \le 1$ for all $z \in \mathbb{C}_\infty \setminus B$; and
$f'(\infty) := \lim _{z \rightarrow \infty} z[f(z) - f(\infty)].$
The analytic capacity of a general subset $E$ of $\mathbb{C}$ is given by: 
\[
\ \gamma (E) = \sup \{\gamma (B) : B\subset E \text{ compact}\}.
\]
Good sources for basic information about analytic
capacity are Chapter VIII of \cite{gamelin}, Chapter V of \cite{conway}, and \cite{Tol14}.

A related capacity, $\gamma _+,$ is defined for subsets $E$ of $\mathbb{C}$ by:
\[
\ \gamma_+(E) = \sup \|\mu \|,
\]
where the supremum is taken over $\mu\in M_0^+(E)$ for which $\|\mathcal{C}(\mu) \|_{L^\infty (\mathbb{C})} \le 1.$ 
Since $\mathcal C\mu$ is analytic in $\mathbb{C}_\infty \setminus \mbox{spt}(\mu)$ and $(\mathcal{C}(\mu)'(\infty) = \|\mu \|$, 
we have:
\[
\ \gamma _+(E) \le \gamma (E)
\]
for all subsets $E$ of $\mathbb{C}$.  X. Tolsa has established the following astounding results. See \cite{Tol03} (also Theorem 6.1 and Corollary 6.3 in \cite{Tol14}) for (1) and (2). See \cite[Proposition 2.1]{Tol02} (also  \cite[Proposition 4.16]{Tol14}) for (3).

\begin{theorem}\label{TolsaTheorem}
(Tolsa 2003)
(1) $\gamma_+$ and $\gamma$ are actually equivalent. 
That is, there is an absolute constant $A_T$ such that 
\begin{eqnarray}\label{GammaEq}
\ \gamma (E) \le A_ T \gamma_+(E)
\end{eqnarray}
for all $E \subset \mathbb{C}.$ 

(2) Semiadditivity of analytic capacity:
\begin{eqnarray}\label{Semiadditive}
\ \gamma \left (\bigcup_{i = 1}^m E_i \right ) \le A_T \sum_{i=1}^m \gamma(E_i)
\end{eqnarray}
where $E_1,E_2,...,E_m \subset \mathbb{C}$ ($m$ may be $\i$).

(3) There is an absolute positive constant $C_T$ such that, for any $a > 0$, we have:  
\begin{eqnarray}\label{WeakOneOne}
\ \gamma(\{\mathcal{C}_*(\nu)  \geq a\}) \le \dfrac{C_T}{a} \|\nu \|.
\end{eqnarray}

\end{theorem}

Combining Theorem \ref{TolsaTheorem} (1), \cite[Theorem 4.14]{Tol14}, and \cite{Tol03} (or \cite[Theorem 8.1]{Tol14}), we get the following corollary. The reader may also see \cite[Corollary 3.1]{acy19}.

\begin{corollary}\label{ZeroAC}
If $\nu\in M_0(\C),$ then there exists 
$\mathcal Q \subset \mathbb{C}$ with $\gamma(\mathcal Q) = 0$ such that $\lim_{\epsilon \rightarrow 0}\mathcal{C} _{\epsilon}(\nu)(z)$ 
exists for $z\in\mathbb{C}\setminus \mathcal Q$.
\end{corollary}

\begin{proposition}\label{GLProp} 
Let $X\subset \mathbb C$ and $a_1,a_2 \in \mathbb C.$ 
Let $\mathcal Q$ be any set in $\C$ with  $\gamma(\mathcal Q) = 0$.
Let $f_1(z)$ and $f_2(z)$ be functions on $\D(\lambda, \delta_0)\setminus \mathcal Q$ for some $\delta_0 > 0.$
 If 
\[  
 \  \lim_{\delta \rightarrow 0} \dfrac{\gamma(\D(\lambda, \delta) \cap X \cap \{|f_i(z) - a_i| > \epsilon\})} {\delta}= 0, ~ i =1,2,
\]
for all $\epsilon > 0,$ then the following properties hold:

(1) 
\[  
 \  \lim_{\delta \rightarrow 0} \dfrac{\gamma(\D(\lambda, \delta) \cap X \cap \{|f_1(z) + f_2(z) - (a_1 + a_2)| > \epsilon\})} {\delta}= 0;
\]

(2) 
\[  
 \  \lim_{\delta \rightarrow 0} \dfrac{\gamma(\D(\lambda, \delta) \cap X \cap \{|f_1(z)f_2(z) - a_1a_2| > \epsilon\})} {\delta}= 0;
\]
(3) If $a_2 \ne 0,$ then
\[  
 \  \lim_{\delta \rightarrow 0} \dfrac{\gamma(\D(\lambda, \delta) \cap X \cap \{|\frac{f_1(z)}{f_2(z)} - \frac{a_1}{a_2})| > \epsilon\})} {\delta}= 0.
\]
 \end{proposition}
 
 \begin{proof}
 We have the following set inclusions.
 \[
 \begin{aligned}
 \ & \{|f_1(z) + f_2(z) - (a_1 + a_2)| > \epsilon\} \\
 \ \subset & \{|f_1(z) - a_1| > \frac{\epsilon}{2}\} \cup \{|f_2(z) - a_2| > \frac{\epsilon}{2}\}
 \end{aligned}
 \]
 and
 \[
 \begin{aligned}
 \ & \{|f_1(z)f_2(z) - a_1a_2| > \epsilon\} \\
 \ \subset & \{|f_2(z)||f_1(z) - a_1| > \frac{\epsilon}{2}\} \cup \{|a_1||f_2(z) - a_2| > \frac{\epsilon}{2}\}\\
 \ \subset & \{\frac 32 |a_2| |f_1(z) - a_1| > \frac{\epsilon}{2}\} \cup \{|f_2(z) - a_2| > \frac 12 |a_2|\} \cup \{|a_1||f_2(z) - a_2| > \frac{\epsilon}{2}\}.
 \end{aligned}
 \]
 (1) and (2) follows from Theorem \ref{TolsaTheorem} (2).	
 
 For (3): From (2), we may assume that $f_1(z) = 1$ and $a_1 = 1.$ Because
 \[
 \begin{aligned}
 \ &  \left \{\left |\frac{1}{f_2(z)} - \frac{1}{a_2})\right | > \epsilon\right \} \\
 \ \subset & \left \{\frac{2}{|a_2|^2}\left |f_2(z) - a_2\right | > \epsilon\right \} \cup \left \{|f_2(z) - a_2| > \frac {|a_2|}{2}\right \}.
 \end{aligned}
 \]
(3) follows from (2) and Theorem \ref{TolsaTheorem} (2).
 \end{proof}

\begin{definition}\label{GLDef}
Let  $\mathcal Q$ be a set with $\gamma(\mathcal Q) = 0$.
Let $f(z)$ be a function defined 
on $\D(\lambda, \delta_0)\setminus \mathcal Q$ for some $\delta_0 > 0.$ The function $f$ has a $\gamma$-limit $a$ at $\lambda$ if
\[  
 \  \lim_{\delta \rightarrow 0} \dfrac{\gamma(\D(\lambda, \delta) \cap \{|f(z) - a| > \epsilon\})} {\delta}= 0
\]
for all $\epsilon > 0$. If in addition, $f(\lambda)$ is well defined and $a = f(\lambda)$, then $f$ is $\gamma$-continuous at $\lambda$. 
\end{definition}

The following corollary follows from Proposition \ref{GLProp}. 

\begin{corollary}\label{GCProp}
If $f(z)$ and $g(z)$ are $\gamma$-continuous at $\lambda,$ then $f(z)+g(z)$ and $f(z)g(z)$ are $\gamma$-continuous at $\lambda.$ If in addition $g(\lambda) \ne 0,$ then $\frac{f(z)}{g(z)}$ is $\gamma$-continuous at $\lambda.$ 
 \end{corollary}

 The following lemma is from \cite[Lemma 3.2]{acy19}. 

\begin{lemma}\label{CauchyTLemma} 
Let $\nu\in M_0(\mathbb{C})$ and assume that for some $\lambda _0$ in $\mathbb C$ we have:
\begin{itemize}
\item[(a)] $\underset{\delta\rightarrow 0}{\overline \lim} \dfrac{| \nu |(\D (\lambda _0, \delta))}{\delta }= 0$ and 
\item[(b)] $\mathcal{C} (\nu)(\lambda _0) = \lim_{\epsilon \rightarrow 0}\mathcal{C} _{\epsilon}(\nu)(\lambda _0)$ exists.
\end{itemize}

Then:

(1) $\mathcal{C}(\nu)(\lambda) = \lim_{\epsilon \rightarrow 0}\mathcal{C} _{\epsilon}(\nu)(\lambda )$ 
exists for every $\lambda$ in $\C$  except perhaps for some set $\mathcal Q$  with analytic capacity $0$.

(2) The  Cauchy transform $\mathcal{C}(\nu)(\lambda)$ is $\gamma$-continuous at $\lambda_0$.
\end{lemma}

For $\nu\in M_0(\C),$ define
\[
\ \mathcal {ND}(\nu) = \left\{\lambda  :\underset{\delta\rightarrow 0}{\overline\lim} \dfrac{|\nu |(\D(\lambda , \delta ))}{\delta} > 0 \right\}
\]
and $\mathcal {ZD}(\nu) = \C \setminus \mathcal {ND}(\nu).$ 
The following lemma is from \cite[Lemma 3.5]{acy19} (also see \cite[Lemma 8.12]{Tol14}).

\begin{lemma}\label{lemmaBasic}
Let $\nu\in M_0(\C)$. Suppose that $|\nu |\perp l_E$ for some Borel set $E\subset \R$. Then
 \[
 \ l_E(\mathcal {ND}(\nu)) = 0.
 \]
\end{lemma}

The following Lemma is from Lemma B in \cite{ars02}. In its statement we let $\area$ denote the area measure
(i.e., two-dimensional Lebesgue measure) on $\mathbb{C}$.

\begin{lemma} \label{lemmaARS}
There are absolute constants $\epsilon _1 > 0$ and $C_1 < \infty$ with the
following property. If $R > 0$ and $E \subset  \overline{\D(0, R)}$ with 
$\gamma(E) < R\epsilon_1$, then
\[
\ |p(\lambda)| \le \dfrac{C_1}{\pi R^2} \int _{\overline{\D(0, R)}\setminus E} |p|\, d\area
\]
for all $\lambda$ in $\D(0, \frac{R}{2})$ and all analytic polynomials $p$.
\end{lemma}

The following lemma is from \cite[Proposition 6.5]{Tol14}.

\begin{lemma}\label{BAForH1Finite}
 Let $\Gamma$ be a rectifiable curve. Let $E\subset \Gamma$ be compact with $l_\Gamma (E) <\infty.$  Let $f : \mathbb C_\i \setminus E\rightarrow \mathbb C$ be analytic such that $\|f\|_\infty \le 1$ and $f(\infty )=0$. Then there is a measure $\nu = b l_\Gamma |_E$, where $b$ is a measurable function with $|b(z)| \le 1$ for all $z\in E$, such that $f(z) = \mathcal C(\nu )(z)$ for all $z \notin E$.
 \end{lemma}
 
 Plemelj's formula for Lipschitz graphs as in X. Tolsa \cite[Theorem 8.8]{Tol14} is applied to the following lemma.
	
\begin{lemma}\label{PFLipschitz}
For all $f\in L^1(l_{\R})$, the non-tangential limits of $\mathcal C(f dz |_{\R})(z)$ exists for $l_{\R}-a.a.$ and the following identities hold $l_{\R}-a.a.$ for $\lambda\in \R$:
	\[
	 \ \dfrac{1}{2\pi i} \lim_{z\in U_n (\lambda, \alpha ) \rightarrow \lambda}\mathcal C(f dz|_{\Gamma})(z) = \dfrac{1}{2\pi i}\mathcal C(f dz|_{\R})(\lambda) + \dfrac{1}{2} f(\lambda),
	\]
	and
	\[
	 \ \dfrac{1}{2\pi i} \lim_{z\in L_n (\lambda, \alpha ) \rightarrow \lambda}\mathcal C(f dz|_{\Gamma})(z) = \dfrac{1}{2\pi i}\mathcal C(f dz|_{\R})(\lambda) - \dfrac{1}{2} f(\lambda).
	\]
\end{lemma}

\section{\textbf{Examples}}

For a compact subset $K\subset \C$ and $\mu\in M_0^+(K),$ the envelope $E = E(K,\mu)$ is the subset of $z\in K$ that satisfies the property: there exists  $\mu _z \in M_0(K)$, with $\mu _z$  absolutely continuous with respect to $\mu$, and such that $\mu_z(\{z\}) = 0$ and $\int f d\mu_z = f(z)$ for each $f \in \text{Rat}(K).$ Chaumat's Theorem says that if $R^\i(K, \mu)$ does not contain a $L^\infty$ summand, then the map $\rho(f) = f$ for $f\in \text{Rat}(K)$ extends an isometric isomorphism and a weak-star homeomorphism from $R^\i(K, \mu)$ onto $R^\i (K, \area |_E)$
(see \cite{cha74} or \cite[Chapter VI.3]{conway}). The map $\rho$ is called Chaumat's map.
We start with a simple version of Main Theorem (1) for $\rikmu$ as the following proposition.

\begin{proposition} \label{REEForSOB}
For a string of beads set $K$ and $\mu\in M_0^+(K),$ assume that $R^\infty(K, \mu)$ has no $L^\i$ summand and the envelope $E$ contains $\Omega.$
Then $\rho(f)\in H^\i (\Omega, \partial_i K)$ for $f\in R^\i (K, \mu).$
\end{proposition}

\begin{proof}
Let $a = \frac{1}{2}i$ and $b = - \frac{1}{2}i$, where $i = \sqrt{-1}$. Let $\omega_a$ and $\omega_b$ be the harmonic measures of $\Omega_U$ and $\Omega_L$ evaluated at $a$ and $b$,
 respectively. Let $\omega = \omega_a + \omega_b.$ 
 It is clear that $\int r(z)(z-a)(z-b) d \omega(z) = 0$ for $r\in Rat(K)$. So $R^\i(K, \omega)$ has no $L^\i$ summand and the envelope $E(K,\omega) \supset \Omega.$ From Chaumat's Theorem, the Chaumat map $\rho$ 
 is an isometric isomorphism and a weak-star homeomorphism from $R^\i(K, \omega)$ onto $R^\i (K, \area_K).$ 
 
Let $\{r_n\}\subset Rat(K)$ be 
such that $r_n \rightarrow f$ in $L^\infty (\omega)$ weak-star topology. Hence, $f\in H^\infty (\Omega_U)$ and $f\in H^\infty(\Omega_L)$. 
Thus, $f_+(z)$ exists, $f_+(z) = f(z),~ \omega_a |_{\partial _i K}-a.a.$; and $f_-(z)$ exists, 
$f_-(z) = f(z),~ \omega_b |_{\partial _i K}-a.a.$. Since $\partial \Omega_U$ and $\partial \Omega_L$ 
are rectifiable Jordan curves, the measures $\omega_a |_{\partial _i K}$, $\omega_b |_{\partial _i K}$, and $l_{\partial _i K}$ are mutually absolutely continuous. Therefore, 
 \[
 \ f_+(z) = f_-(z) = f(z),~ l_{\partial _i K}-a.a.
 \]
 Hence, $R^\i (K, \area_K) \subset H^\i (\Omega, \partial_i K).$ Now the proof follows from Chaumat's Theorem.
\end{proof}

However, for $1\le t <\infty$, the situations for $R^t(K, \mu)$ are very different. The following example constructs a measure $\mu$ such that $S_\mu$ on $R^2(K, \mu)$ is pure, $\text{abpe}(R^2(K, \mu)) = \Omega$, and $R^2(K, \mu)$ splits.

\begin{example}\label{SOBExample1}
Let $\area_W$ be the weighted planar Lebesgue measure $W(z)d\area(z)$ on $\Omega$, where
 \[
 \ W(z) = \begin{cases}  \text{dist}(z, \mathbb R)^4, & z \in \Omega_U\cup \Omega_L;\\0. & z \in (\Omega_U\cup \Omega_L)^c.  \end{cases}
 \]
Then,
 \begin{eqnarray}\label{SOBEq1}
 \ R^2(K, \area_W) = R^2(K_U, \area_W |_{\Omega_U}) \oplus R^2(K_L, \area_W |_{\Omega_L}).
 \end{eqnarray}  
\end{example} 

\begin{proof}
It is easy to verify that $S_{\area_W}$ is pure and $\text{abpe}(R^2(K, \area_W)) = \Omega$.
Let $\lambda\in \partial_i K$. Choose a subsequence $\{\lambda_{n_k}\}$ such that $\lambda_{n_k}\rightarrow \lambda$. For $g\perp R^2(K, \area_W)$, we have
 \[
 \ \begin{aligned}
 \ & \int \left | \dfrac{1}{z-\lambda_{n_k}} - \dfrac{1}{z-\lambda} \right | |g|d\area_W \\
 \ \le &|\lambda_{n_k} - \lambda| \left(\int _G \dfrac{W(z)}{|z-\lambda_{n_k}|^2 |z - \lambda|^2}d\area(z) \right )^{\frac{1}{2}} \|g\|_{L^2(\area_W)} \\
 \ \le &\sqrt{\pi}\|g\|_{L^2(\area_W)}|\lambda_{n_k} - \lambda| \\
 \ & \rightarrow 0.
 \ \end{aligned} 
\]
Thus, $\mathcal C(g\area_W)(\lambda) = 0$ for $\lambda\in \partial _i K$. Let $\nu$ be $\dfrac{1}{2\pi i}dz$ restricted to upper circle of $|z| = 2$ and the interval $[-2,2]$. By Fubini's theorem, we see that 
 \[
 \ \int \mathcal C(\nu)(z) g(z) d\area_W(z) = - \int \mathcal C(g\area_W)(\lambda) d\nu (\lambda) = 0.
 \]
Hence, $\chi_{\Omega_U}(z) \in R^2(K, \area_W)$, where $\chi_{A}$ is the characteristic function for the set $A$. \eqref{SOBEq1} is proved.   
\end{proof}
\smallskip

Now let us look at a simple example of $\mu$ such that $S_\mu$ is irreducible.

\begin{example}\label{SOBExample2}
Let $a,$ $b,$ $\omega_a,$ $\omega_b,$ and  $\omega$ be as in the beginning of the proof of Proposition \ref{REEForSOB}.  Then, for every $f\in R^2(K, \omega)$, the non-tangential limit from the top on $\partial _i K$,
  $f_+(z)$,  and the non-tangential limit from 
  the 
  bottom on $\partial _i K$, $f_-(z)$, exist $l_{\partial _i K}-a.a.$, and $f_+(z) = f_-(z) = f(z),~ l_{\partial _i K}-a.a..$ Consequently, $S_\omega$ is irreducible.     
\end{example}

\begin{proof} The same argument in the proof of Proposition \ref{REEForSOB} applies.
\end{proof}

The examples above 
suggest that, in order to ensure that $S_\mu$ on $R^t(K, \mu)$ is irreducible, 
the 
values of every function $f\in R^t(K, \mu)$ on $\Omega_U$ and $\Omega_L$ must be associated in some ways. Main Theorem points out the relationship. 

Assuming our Main Theorem has been proved, we have the following proposition that combines Example \ref{SOBExample1} and \ref{SOBExample2}.

\begin{proposition}\label{SOBRemovable}
Let $K$ be a string of beads set.
If $F\subset \partial_i K$ is a compact subset with $l_{\partial_i K}(F) > 0$, then there exists $\mu \in M_0^+(K)$ such that $S_\mu$ on $R^2(K, \mu)$ is irreducible, $\text{abpe}(R^2(K, \mu)) = \Omega$, and $E_0 = F.$ 

\end{proposition}

\begin{proof}
Let $\Gamma_U$ be a closed rectifiable Jordan curve in $K_U$ such that $\Gamma_U\cap \partial \Omega_U = F$. Let $\Gamma_L$ be a closed  rectifiable Jordan curve in $K_L$ such that $\Gamma_L\cap \partial \Omega_L = F$. Let $G_U$ and $G_L$ be the regions surrounded by $\Gamma_U$ and $\Gamma_L$, respectively. Let $\omega_a$ and $\omega_b$ be the harmonic measures of $G_U$ and $G_L$, respectively, where $a\in G_U$ and $b\in G_L$. Let $\area_W$ be the weighted planar Lebesgue measure as in Example \ref{SOBExample1}. Let $\mu = \omega_a + \omega_b + \area_W$. 

Similarly to the proof of Example \ref{SOBExample2}, we see that $f_+(z) = f_-(z) = f(z)$ on $F$ $l_{\partial_i K}-a.a.$. Hence, $S_\mu$ is irreducible. Obviously, $\text{abpe}(R^2(K, \mu)) = \Omega$.

Suppose there exists a compact subset $F_0\subset \partial_i K\setminus F$ that is a part of $E_0$ and $l_{\partial_i K}(F_0) > 0$.  Let $f$ be a bounded analytic function on $\C_\i\setminus F_0$ with $\|f\|_\infty = 1.$  Using Lemma \ref{BAForH1Finite}, we have $f = \mathcal C(wl_{\partial_i K})$, where $w$ is bounded and supported in $F_0$. 
Let $\{g_j\}$ be a dense subset of $R^2(K, \mu) ^\perp.$
As in the proof 
of Example \ref{SOBExample1}, we see that $\dfrac{1}{z-\lambda} \in R^2(K, \mu)$ for $\lambda \in F_0$. Moreover,
 \[
 \ \begin{aligned}
 \ & \int \dfrac{1}{|z-\lambda|} |g_j|d\mu \\
 \ \le & \int _{\Gamma_U}\dfrac{1}{|z-\lambda|} |g_j|d\omega_a + \int _{\Gamma_L}\dfrac{1}{|z-\lambda|} |g_j|d\omega_b + \int _\Omega\dfrac{W(z)}{|z-\lambda|} |g_j|d\area \\
 \ \le & \dfrac{\|g_j\|}{\text{dist}(\lambda, \Gamma_U)} + \dfrac{\|g_j\|}{\text{dist}(\lambda, \Gamma_L)} + \|g_j\| \\
 \ \end{aligned}
 \]
for $\lambda\in F_0$. Using Fubini's theorem, we have
 \[
 \ \int f(z) g_j(z) d\mu(z) = - \int \mathcal C( g_j\mu)(\lambda) w(\lambda) d l_{\partial_i K}(\lambda) = 0. 
 \]
Hence, $f\in R^2(K,\mu)$. This contradicts 
 Main Theorem  
because, by Lemma \ref{PFLipschitz}, $f_+$ and $f_-$ are not identically equal with respect to $l_{\partial_i K}$ on $F_0$.
\end{proof}

\section{\textbf{Construction of $E_0$ and proof of Main Theorem (1)}}

We assume that $1 \le t < \infty,$ $K$ is a string of beads set, $\mu\in M_0^+(K),$ $S_\mu$ on $R^t(K,\mu)$ is pure, and $\text{abpe}(R^t(K,\mu)) = \Omega (= \text{int}(K)).$ 

Let $\mu  = hl_{\partial_i K}  + \mu_s$ be the Radon-Nikodym decomposition with respect to $l_{\partial_i K}$ on $\mathbb R$, where $h \in L^1(l_{\partial_i K})$ and $\mu_s \perp  l_{\partial_i K}$. We define $\mathcal N(h) = \{z\in \partial_i K, ~ h(z) \ne 0\}$ and $\mathcal Z(h) = \{z\in \partial_i K, ~ h(z) = 0\}.$ 

Let $\{g_j\} \subset R^t(K,\mu)^\perp$ be a dense set. For $f\in R^t(K, \mu)$, we see that $f(z)$ is analytic on $\text{abpe}(R^t(K,\mu)) = \Omega$ and $\frac{f(z)-f(\lambda)}{z-\lambda}\in R^t(K,\mu)$ for $\lambda\in \Omega.$ Therefore, by Corollary \ref{ZeroAC},
 \begin{eqnarray}\label{BasicEqn1}
 	f(\lambda) \mathcal C(g_j\mu) (\lambda) = \mathcal C(fg_j\mu) (\lambda), ~ \lambda\in \Omega, ~ \gamma-a.a.
 \end{eqnarray}
 for all $j\ge 1.$
 
 \begin{lemma}\label{NTLimitLemma}
Let $f(z)$ be analytic on $U_{n_\lambda}(\lambda, \alpha)$ for some $\lambda \in \partial_i K$ and $n_\lambda \ge 1$ such that \eqref{ULConeDefEq3} holds. If for all $\epsilon > 0$, 
\[
 \ \lim_{n\rightarrow \infty} n\gamma(U_{n}(\lambda, \alpha)\cap \{|f(z) - a| > \epsilon\}) = 0,
 \]
then for $0 < \theta < 1$,
 \[
 \ \underset{z\in U_{n}(\lambda, \theta \alpha)}{\lim_{n\rightarrow \infty}}f(z) = a.
 \]
 The same result holds if we replace $U_{n}(\lambda, \alpha)$ by $L_{n}(\lambda, \alpha).$
\end{lemma}

\begin{proof} Let $n > n_\lambda.$
There exists a constant $0 < \theta_0 < 1$ only depending on $\theta$ such that $\D(w, \theta_0\frac 1n)\subset U_n(\lambda, \alpha)$ for $w\in U_{2n}(\lambda, \theta \alpha)$ and $\frac{1}{2(n+1)} < |w - \lambda | \le \frac{1}{2n}$. From the assumption, we find $n_\epsilon$ such that for $n > n_\epsilon,$ 
 \[
 \ \gamma(U_n(\lambda, \alpha)\cap \{|f(z) - a| > \epsilon\}) < \epsilon_1 \theta_0\frac 1n,
 \]
where $\epsilon_1$ is as in Lemma \ref{lemmaARS}. Applying  Lemma \ref{lemmaARS}, we conclude that
 \[
 \ |f(w) - a| \le \dfrac{C_1n^2}{\pi}\int _{\D(w,\theta_0\frac 1n)\setminus \{|f(z) - a| > \epsilon\}}|f(z) - a| d\area (z) \le C_{2} \epsilon 
 \]
for $w\in U_{2n}(\lambda, \theta \alpha)$ and $\frac{1}{2(n+1)} < |w - \lambda | \le \frac{1}{2n}$.
The lemma follows.
\end{proof}
 
 \begin{lemma}\label{CTDefinedOnE}
 The principle value $\lim_{\epsilon \rightarrow 0} \mathcal C_\epsilon (g\mu) (z) = \mathcal C(g\mu) (z)$ exists $l_{\partial_i K}-a.a.$ for $g\in L^1(\mu)$. 
 \end{lemma}

\begin{proof}
The proof follows from Corollary \ref{ZeroAC} and the fact that $\gamma |_{\partial_i K} \approx l_{\partial_i K}$ (see \cite[VIII.2.2]{gamelin}).	
\end{proof}

 \begin{definition}\label{FRSOBDef1}
 Define
  \[
  \ F_0(\{g_j\mu\}) = \bigcap_{j=1}^\infty \{z\in \mathcal Z(h),~ \mathcal C(g_j\mu)(z) = 0\}
  \]
  and 
  \[
  \ R_0(\{g_j\mu\}) = \bigcup_{j=1}^\infty \{z\in \mathcal Z(h),~ \mathcal C(g_j\mu)(z) \ne 0\}
  \]
  \end{definition}
  
  Clearly, from Lemma \ref{CTDefinedOnE}, we see that
  \begin{eqnarray}\label{FREq1}
  \begin{aligned}
  \ & l_{\partial_i K}(F_0(\{g_j\mu\})\cap R_0(\{g_j\mu\})) =0, \\
  \ & l_{\partial_i K}( \mathcal Z(h) \setminus (F_0(\{g_j\mu\})\cup R_0(\{g_j\mu\}))) = 0.
  \end{aligned}
  \end{eqnarray} 
  
  \begin{lemma}\label{NTLimit1}
  If $f\in R^t(K,\mu),$ then
  \[
  \ f_+(z) = f_-(z), ~ z \in R_0(\{g_j\mu\}),~ l_{\partial_i K}-a.a..
  \]
  \end{lemma}
  
  \begin{proof}
  Let $\lambda \in R_0(\{g_j\mu\}).$ We find $j_0$ such that $\lim_{\epsilon \rightarrow 0} \mathcal C_\epsilon (g_{j_0}\mu) (\lambda) = \mathcal C(g_{j_0}\mu) (\lambda)$ exists and $C(g_{j_0}\mu) (\lambda) \ne 0.$ Since $l_{\partial_i K}|_{\mathcal Z(h)}\perp |g_{j_0}\mu|,$ by Lemma \ref{lemmaBasic}, we see that 
  \[
  \ l_{\partial_i K}(\mathcal Z(h)\cap (\mathcal {ND}(g_{j_0}\mu) \cup \mathcal {ND}(fg_{j_0}\mu))) = 0.
  \] 
  	Hence, we may assume that $\lambda \in \mathcal {ZD}(g_{j_0}\mu) \cap \mathcal {ZD}(fg_{j_0}\mu).$ From Lemma \ref{CauchyTLemma}, we infer that $\mathcal C(g_{j_0}\mu)$ and  $\mathcal C(fg_{j_0}\mu)$ are $\gamma$-continuous at $\lambda.$ Since $\mathcal C(g_{j_0}\mu)(\lambda)\ne 0,$ by Corollary \ref{GCProp}, we see that $\frac{\mathcal C(fg_{j_0}\mu)(z)}{\mathcal C(g_{j_0}\mu)(z)}$ is $\gamma$-continuous at $\lambda.$ From \eqref{BasicEqn1} and Lemma \ref{NTLimitLemma}, we conclude that
  	\[
  	\ f_+(\lambda) = f_-(\lambda) = \frac{\mathcal C(fg_{j_0}\mu)(\lambda)}{\mathcal C(g_{j_0}\mu)(\lambda)}.
  	\] 
  	The lemma is proved. 
  \end{proof}

\begin{lemma}\label{GPTheorem}
Let $\nu\in M_0(\mathbb{C})$. Suppose that $\nu = h_0l_{\partial_i K} + \sigma$ is the Radon-Nikodym 
decomposition with respect to $l_{\partial_i K}$, where $h_0\in L^1(l_{\partial_i K})$ and $\sigma\perp l_{\partial_i K}$. Then there exists a subset $\mathcal Q\subset \mathbb C$ with $\gamma(\mathcal Q) = 0$, such that the following hold:

(a) $\mathcal C(\nu ) (\lambda) = \lim_{\epsilon\rightarrow 0} \mathcal C_{\epsilon}(\nu)(\lambda)$ exists for $\lambda\in \mathbb C\setminus \mathcal Q$,

(b) for $\lambda \in \partial_i K \setminus \mathcal Q$ and $\epsilon > 0$, $v^+(\nu)(\lambda):= \mathcal{C}(\nu )(\lambda ) +\pi i h_0(\lambda)$,
\begin{eqnarray}\label{GPTheoremEq1}
 \ \lim_{n \rightarrow \infty} n\gamma(U_n(\lambda, \alpha)\cap \{ |\mathcal{C}(\nu )(z ) - v^+(\nu)(\lambda)| > \epsilon \}) = 0, 
 \end{eqnarray}
and

(c) for $\lambda \in \partial_i K \setminus \mathcal Q$ and $\epsilon > 0$, $v^-(\nu)(\lambda):= \mathcal{C}(\nu )(\lambda ) -\pi i h_0(\lambda)$,
\begin{eqnarray}\label{GPTheoremEq2}
 \ \lim_{n \rightarrow \infty} n\gamma(L_n(\lambda, \alpha)\cap \{ |\mathcal{C}(\nu )(z ) - v^-(\nu)( \lambda)| > \epsilon \}) = 0. 
 \end{eqnarray}
\end{lemma}

\begin{proof}
(a) is trivial by Corollary \ref{ZeroAC}. For (b): Because
\[
\begin{aligned}
\ & \{ |\mathcal{C}(\nu )(z ) - v^+(\nu)(\lambda)| > \epsilon \} \\
\ \subset & \{ |\mathcal{C}(\sigma)(z ) - \mathcal{C}(\sigma)(\lambda)| > \frac{\epsilon}{2} \} \cup \{ |\mathcal{C}(h_0l_{\partial_i K} )(z ) - v^+(h_0l_{\partial_i K})(\lambda)| > \frac{\epsilon}{2} \},
\end{aligned}
\]
(b) follows from Lemma \ref{CauchyTLemma}, Lemma \ref{PFLipschitz} and Theorem \ref{TolsaTheorem} (2). The proof of (c) is the same.	
\end{proof}

\begin{definition}\label{FRSOBDef2}
 Define
  \[
  \ F_+(\{g_j\mu\}) = \bigcap_{j=1}^\infty \{z\in \mathcal N(h),~ v^+(g_j\mu)(z) = 0\},
  \]
  \[
  \ F_-(\{g_j\mu\}) = \bigcap_{j=1}^\infty \{z\in \mathcal N(h),~ v^-(g_j\mu)(z) = 0\},
  \]
  and 
  \[
  \begin{aligned}
  \ R_1(\{g_j\mu\}) =& \left (\bigcup_{j=1}^\infty \{z\in \mathcal N(h),~ v^+(g_j\mu)(z) \ne 0\} \right ) \\
  \ & \bigcap \left (\bigcup_{j=1}^\infty \{z\in \mathcal N(h),~ v^-(g_j\mu)(z) \ne 0\} \right ).
   \end{aligned}
  \]
  \end{definition}
  
  Clearly, from Lemma \ref{CTDefinedOnE}, we see that
  \begin{eqnarray}\label{FREq2}
  \begin{aligned}
  \ & l_{\partial_i K}((F_+(\{g_j\mu\}) \cup F_-(\{g_j\mu\}))\cap R_1(\{g_j\mu\})) =0, \\
  \ & l_{\partial_i K}( \mathcal N(h) \setminus (F_+(\{g_j\mu\}) \cup F_-(\{g_j\mu\}))\cup R_1(\{g_j\mu\})))) =0.
  \end{aligned}
  \end{eqnarray} 
  
  \begin{lemma}\label{NTLimit2}
  If $f\in R^t(K,\mu),$ then
  \[
  \ f(z) \mathcal C(g_j\mu)(z) = \mathcal C(fg_j\mu)(z), ~ z \in \mathcal N(h),~ l_{\partial_i K}-a.a..
  \]
  \end{lemma}
  
  \begin{proof}
  There exists a sequence $\{f_n\}\subset Rat(K)$ such that
  \[
  \ \|f_n - f\|_{L^t(\mu)}\rightarrow 0,~ f_n \rightarrow  f, ~\mu-a.a..
 \]
 
 Let $g\in \{g_j\}.$ From Corollary \ref{ZeroAC}, we let $\mathcal Q_1\subset \C$ with $\gamma(\mathcal Q_1) = 0$ such that the principal values of $\mathcal{C}(g\mu)(z ),$ $\mathcal{C}(f_ng\mu)(z )$ for $n\ge 1,$ and $\mathcal{C}(fg\mu)(z )$ exist for $z\in \C\setminus \mathcal Q_1$. Define
 \[
 \ A_{nm} = \left\{z\in \C\setminus \mathcal Q_1:~ |\mathcal{C}(f_ng\mu)(z ) - \mathcal{C}(fg\mu)(z )|\ge \dfrac{1}{m} \right\}.
 \]
Since
 \[
 \begin{aligned}
 \ |\mathcal{C}(f_ng\mu)(z ) - \mathcal{C}(fg\mu)(z )| & = \lim_{\epsilon \rightarrow 0} | \mathcal{C} _{\epsilon}(f_ng\mu)(z ) - \mathcal{C} _{\epsilon}(fg\mu)(z )|\\
& \le \mathcal{C} _{*}((f_n - f)g\mu)(z ),
 \end{aligned}
 \]
applying Theorem \ref{TolsaTheorem} (3), we get
 \[
 \ \gamma (A_{nm}) \le \gamma \left\{\mathcal{C} _{*}((f_n - f)g\mu)(z ) \ge \dfrac{1}{m} \right\} \le C_Tm \|g\|\|f_n-f\|_{L^t(\mu)}.  
 \]
Choose $n_m$ so that $\|f_{n_m}-f\|_{L^t(\mu)} \le \frac{1}{m\|g\|2^m}$ and we have
$\gamma (A_{n_mm}) \le \frac{C_T}{2^m}$.
Set $B_k = \cup_{m=k}^\infty A_{n_mm}$.
Applying Theorem \ref{TolsaTheorem} (2), we have
 \[
 \ \gamma (B_k\cup \mathcal Q_1) \le A_T\sum_{m=k}^\infty \gamma (A_{n_mm})\le \dfrac{2A_TC_T}{2^k}.
 \]
Set $\mathcal Q_2 = \cap_{k=1}^\infty B_k$. Clearly, $\gamma (\mathcal Q_2) = 0$. If $\mathcal Q=\mathcal Q_1\cup \mathcal Q_2$, then on $\C\setminus \mathcal Q$, $\mathcal C(f_{n_m}g\mu)(\lambda )$ converges to $\mathcal C(fg\mu)(\lambda )$.
  Therefore, 
   \[
  \ f_{n_m}(z)\mathcal C(g\mu)(z) = \mathcal C(f_{n_m}g\mu)(z) \rightarrow \mathcal C(fg\mu)(z), ~\gamma-a.a..
  \]
  The lemma follows since $\gamma |_{\partial_i K} \approx l_{\partial_i K}.$	
  \end{proof}

  \begin{lemma}\label{NTLimit3}
  	If $f\in R^t(K,\mu),$ then
  \[
  \ f_+(z) = f_-(z) = f(z), ~ z \in R_1(\{g_j\mu\}),~ l_{\partial_i K}-a.a..
  \]
  \end{lemma}
  
  \begin{proof}
  Let $\lambda \in R_1(\{g_j\mu\}).$ We find $j_1$ and $j_2$ such that $v^+ (g_{j_1}\mu) (\lambda)$ and $v^- (g_{j_2}\mu) (\lambda)$ exist, $v^+ (g_{j_1}\mu) (\lambda) \ne 0,$ and $v^- (g_{j_2}\mu) (\lambda) \ne 0.$ From \eqref{GPTheoremEq1} and \eqref{GPTheoremEq2}, we assume that 
  \[
 \ \lim_{n \rightarrow \infty} n\gamma\left (U_n(\lambda, \alpha)\cap \{ |\mathcal{C}(g_{j_1}\mu )(z ) - v^+(g_{j_1}\mu)(\lambda)| > \epsilon \}\right ) = 0, 
 \]
 \[
 \ \lim_{n \rightarrow \infty} n\gamma\left (U_n(\lambda, \alpha)\cap \{ |\mathcal{C}(fg_{j_1}\mu )(z ) - v^+(fg_{j_1}\mu)(\lambda)| > \epsilon \}\right ) = 0 
 \]
 \[
 \ \lim_{n \rightarrow \infty}n\gamma\left (L_n(\lambda, \alpha)\cap \{ |\mathcal{C}(g_{j_2}\mu )(z ) - v^-(g_{j_2}\mu)(\lambda)| > \epsilon \}\right ) = 0, 
 \]
 and
 \[
 \ \lim_{n \rightarrow \infty}n\gamma\left (L_n(\lambda, \alpha)\cap \{ |\mathcal{C}(fg_{j_2}\mu )(z ) - v^-(fg_{j_2}\mu)(\lambda)| > \epsilon \}\right ) = 0. 
 \]
 Applying Proposition \ref{GLProp} (3) for $X = U_{n_\lambda}(\lambda, \alpha)$ for some $n_\lambda \ge 1,$ we conclude that
 \[
 \ \lim_{n \rightarrow \infty} n\gamma\left (U_n(\lambda, \alpha)\cap \left \{\left  |\dfrac{\mathcal{C}(fg_{j_1}\mu )(z )}{\mathcal{C}(g_{j_1}\mu )(z )} - \dfrac{v^+(fg_{j_1}\mu)(\lambda)}{v^+(g_{j_1}\mu)(\lambda)} \right | > \epsilon \right\}\right ) = 0. 
 \]
 Since by \eqref{BasicEqn1}, $f(z) = \frac{\mathcal{C}(fg_{j_1}\mu )(z )}{\mathcal{C}(g_{j_1}\mu )(z )}$ for $z\in U_{n_\lambda}(\lambda, \alpha),~\gamma-a.a.$ and by Lemma \ref{NTLimit2}, $f(z) =\frac{v^+(fg_{j_1}\mu)(z)}{v^+(g_{j_1}\mu)(z)},~ z\in \mathcal N(h)\cap\{z:~v^+(g_{j_1}\mu)(z) \ne 0\}, ~l_{\partial_i K}-a.a.,$ we get
 \[
 \ \lim_{n \rightarrow \infty} n\gamma\left (U_n(\lambda, \alpha)\cap \{ |f(z ) - f(\lambda) | > \epsilon \} \right ) = 0. 
 \]
 Thus, from Lemma \ref{NTLimitLemma}, we conclude that $f_+(\lambda) = f(\lambda).$ Similarly, $f_-(\lambda) = f(\lambda).$ The lemma is proved.
  \end{proof}
  
 \begin{proof} (Main Theorem (1))
 Let $E_0 = R_0(\{g_j\mu\})\cup R_1(\{g_j\mu\}).$ Then the proof follows from Lemma \ref{NTLimit1} and Lemma \ref{NTLimit3}.	
 \end{proof}

\section{\textbf{Building block functions and an integral estimate}}

We will continue to use the assumptions and notations of the last section. Set
\[
\ F(\{g_j\mu\}) = F_0(\{g_j\mu\})\cup F_+(\{g_j\mu\}) \cup F_-(\{g_j\mu\}).
\]
Clearly, 
\[
\ l_{\partial_i K}(E_0\cap F(\{g_j\mu\})) = 0,~ l_{\partial_i K}(\partial_i K\setminus (E_0\cup F(\{g_j\mu\}))= 0.
\] 
A measure $\eta \in M_0^+(\C)$ is $c$-linear growth if
\[
\ \eta(\D(\lambda, \delta)) \le c \delta,\text{ for }\lambda\in \C \text{ and }\delta > 0.
\]
The following lemma is critical for proving Main Theorem (2).

\begin{lemma} \label{BBFunctions} 
Let $E_1\subset F(\{g_j\mu\})$ be a compact subset with $\gamma(E_1) > 0$. Then 
there exist $f\in R^{t,\i}(K,\mu),$ $\eta\in M_0^+(E_1),$ and $\|\mathcal C_\epsilon (\eta) \| \le C_3$ such that $f(z) = \mathcal C(\eta)(z)$ for $z\in E_1^c$, 
 \[
 \ \|f\|_{L^\infty(\mu)} \le C_3,~ f(\infty) = 0,~ f'(\infty) = \gamma(E_1),
 \]
and $\rho(f) = \mathcal C(\eta)|_\Omega \in H^\i (\Omega, E_0).$ 
\end{lemma}

\begin{proof}
From Theorem \ref{TolsaTheorem} (2), we have
\[
\ \gamma(E_1) \le A_T(\gamma(E_1 \cap F_0(\{g_j\mu\})) + \gamma(E_1 \cap F_+(\{g_j\mu\})) + \gamma(E_1 \cap F_-(\{g_j\mu\}))).
\]
Let $E_1'$ denote one of the sets $E_1 \cap F_0(\{g_j\mu\}),$ $E_1 \cap F_+(\{g_j\mu\}),$ or $E_1 \cap F_-(\{g_j\mu\})$ that has the biggest analytic capacity. Then $\gamma(E_1') \ge \frac{1}{3A_T} \gamma(E_1).$	

Let $A_j = \{\mathcal C_*(g_j\mu)(z) \le M_j\}$ for some constant $M_j > 0.$ By Theorem \ref{TolsaTheorem} (3), we can select $M_j$ so that $\gamma(A_j^c) < \frac{\gamma(E_1)}{2^j(6A_T^3)}$. Set $E_2 = \cap_{j=1}^\infty A_j \cap E_1'$. Then applying Theorem \ref{TolsaTheorem} (2), we get
\[
 \ \gamma (E_2) \ge \dfrac{1}{A_T}\gamma (E_1') - \gamma(\cup_{j=1}^\infty A_j^c) \ge \dfrac{\gamma(E_1)}{3A_T^2} - A_T\sum_{j=1}^\infty \gamma(A_j^c) \ge \dfrac{\gamma(E_1)}{6A_T^2}
 \] 
and 
\begin{eqnarray}\label{BBFunctionsEq1}
\ \mathcal C_*(g_j\mu)(z) \le M_j,\text{ for } z\in E_2\text{ and }j \ge 1.
\end{eqnarray}

From Theorem \ref{TolsaTheorem} (1) and \cite[Theorem 4.14]{Tol14}, there exists $\eta_0\in M_0^+(E_2)$ such that $c_1 \gamma (E_2) \le \|\eta_0\|,$ $\eta_0$ is $1$-linear growth, and $\CT_* (\eta_0) \le 1.$ Clearly, $\eta_0$ is absolutely continuous with respect to $l_{E_2}.$ The Cauchy transform of $l_{E_2}$ operator is $L^2(l_{E_2})$ bounded (see \cite{cmm82} or \cite[Theorem 3.11]{Tol14}). Hence, from \cite{tol98} (or \cite[Theorem 8.1]{Tol14}), we have, for $j \ge 1,$
\begin{eqnarray}\label{BBFunctionsEq2}
\ \lim_{\epsilon\rightarrow 0} \mathcal C_\epsilon (g_j\mu)(z) = \mathcal C (g_j\mu)(z), ~ \eta_0-a.a..
\end{eqnarray} 

There exists a subsequence $f_k(z) = \mathcal C_{\epsilon_k}(\eta_0)(z)$ such that $f_k$ converges to $f_0\in L^\infty(\mu)$ in weak-star topology, $f_k(\lambda) $ converges to $f_0(\lambda) = \mathcal C(\eta_0)(\lambda)$ uniformly on any compact subset of $(\text{spt}(\eta_0))^c$ as $\epsilon_k\rightarrow 0.$ Since
\begin{eqnarray}\label{BBFunctionsEq3}
 \ \int \mathcal C_{\epsilon_k}(\eta_0)(z) g_j(z)d\mu (z) = - \int \mathcal C_{\epsilon_k}(g_j\mu) (z) d\eta_0 (z),
 \end{eqnarray}
 Using \eqref{BBFunctionsEq1}, \eqref{BBFunctionsEq2}, and the Lebesgue dominated convergence theorem, we get
 \begin{eqnarray}\label{BBFunctionsEq4}
 \ \int f_0(z) g_j(z)d\mu (z) = - \int \mathcal C(g_j\mu) (z) d\eta_0 (z).
 \end{eqnarray}
Let $\lambda\notin E_2$ and $d = \text{dist}(\lambda, E_2)$,
for $z\in \D(\lambda, \frac{d}{2})$ and $\epsilon < \frac{d}{2}$, we have
 \[
 \ \left |\dfrac{\mathcal C_\epsilon (\eta_0)(z) - f_0(\lambda)}{z - \lambda} \right |\le \left |\mathcal C_\epsilon  \left (\dfrac{\eta_0 (s)}{s - \lambda} \right ) (z) \right | \le \dfrac{2}{d^2} \|\eta_0\|. 
 \]
For $z\notin \D(\lambda, \frac{d}{2})$ and $\epsilon < \frac{d}{2}$,
 \[
 \ \left |\dfrac{\mathcal C_\epsilon (\eta_0)(z) - f_0(\lambda)}{z - \lambda} \right | \le \dfrac{4}{d}.
 \]
Thus, we can replace the above proof for the measure $\dfrac{\eta_0 (s)}{s - \lambda}$. In fact, we can choose a subsequence  $\{\mathcal C_{\epsilon_{k_m}} (\eta_0)\}$ such that $e_{k_m}(z) = \dfrac{\mathcal C_{\epsilon_{k_m}} (\eta_0)(z) - f_0(\lambda)}{z - \lambda}$ converges to $e(z)$ in weak-star topology. Clearly, $(z-\lambda)e_{k_m}(z)  + f_0(\lambda) = \mathcal C_{\epsilon_{k_m}} (\eta_0)(z)$ converges to $(z-\lambda)e(z)  + f_0(\lambda) = f_0(z)$ in weak-star topology.  
On the other hand, the equation \ref{BBFunctionsEq3} becomes
 \begin{eqnarray}\label{BBFunctionsEq5}
 \  \int \mathcal C_{\epsilon_{k_m}}(\dfrac{\eta_0(s)}{s-\lambda})(z) g_jd\mu  = - \int \mathcal C_{\epsilon_{k_m}}(g_j\mu)(z) \dfrac{d\eta_0(z)}{z-\lambda}
 \end{eqnarray}
and for $\epsilon_{k_m} < \frac{d}{2}$, we have
 \begin{eqnarray}\label{BBFunctionsEq6}
 \ \begin{aligned}
\ & \left | \mathcal C_{\epsilon_{k_m}}(\dfrac{\eta_0 (s)}{s-\lambda})(z) - e_{k_m}(z) \right |\\
\ \le &\begin{cases}0, & z\in \D(\lambda, \frac{d}{2}), \\ \dfrac{2}{d^2}\eta_0 (\D(z, \epsilon_{k_m})), & z\notin \D(\lambda, \frac{d}{2})\end{cases},
\ \end{aligned}
 \end{eqnarray}
which goes to zero as $\epsilon_{k_j} \rightarrow 0$. Combining \eqref{BBFunctionsEq1}, \eqref{BBFunctionsEq2}, \eqref{BBFunctionsEq5}, \eqref{BBFunctionsEq6}, and the Lebesgue dominated convergence theorem, 
 \begin{eqnarray}\label{BBFunctionsEq7}
 \ \int \dfrac{f_0(z) - f_0(\lambda)}{z - \lambda} g_j(z)d\mu (z) = - \int \mathcal C(g_j\mu) (z) \dfrac{d\eta_0 (z)}{z - \lambda}.
 \end{eqnarray}
 Using \cite{tol98} (or \cite[Theorem 8.1]{Tol14}), we see that
 \[
 \ f_0(z) = \mathcal C(\eta_0)(z), ~ \mu |_{\mathcal N(h)}-a.a..
 \] 
 
 Case I: $E_2 \subset E_1 \cap F_0(\{g_j\mu\}),$ then from \eqref{BBFunctionsEq4} and \eqref{BBFunctionsEq7}, we see that, since $\{g_j\}$ is dense in $R^t(K,\mu)^\perp,$ $f_0\in R^{t,\i}(K,\mu)$ and $\frac{f_0(z) - f_0(\lambda)}{z - \lambda}\in R^{t,\i}(K,\mu),$ which implies $\rho(f_0) = f_0|_\Omega \in H^\i (\Omega, E_0).$ Set 
 \[
 \ f = \dfrac{f_0}{\|\eta_0\|} \gamma (E_1) \text{ and } \eta  = \dfrac{\eta_0}{\|\eta_0\|} \gamma (E_1). 
 \]
 
 Case II: $E_2 \subset E_1 \cap F_+(\{g_j\mu\}),$ then from \eqref {BBFunctionsEq4} , we get, since $v^+(g_j\mu)(z) = 0, ~z \in E_2,~ l_{\partial_i K}-a.a.,$
 \[
 \ \int f_0(z) g_j(z)d\mu (z) = - \int \mathcal C(g_j\mu) (z) d\eta_0 (z) = \pi i \int h(z)g_j(z) d\eta_0(z).
 \]
 The function $f_0 - \pi i \frac{d\eta_0}{dl_{\partial_i K}}$ is the non-tangential limit of $f_0$ from the bottom on $E_2$ (Lemma \ref{PFLipschitz}). Thus, 
 $f_0 - \pi i \frac{d\eta_0}{dl_{\partial_i K}}\in L^\infty(\mu).$ Hence, 
 \[
 \ \int (f_0(z) - \pi i \dfrac{d\eta_0}{dl_{\partial_i K}}(z)) g_j(z)d\mu (z) = 0,
 \]
 which implies $f_0 - \pi i \frac{d\eta_0}{dl_{\partial_i K}}\in R^{t,\i}(K,\mu).$
 
 Using a similar argument  with
 \eqref{BBFunctionsEq7} instead of  \eqref{BBFunctionsEq4},
 we see that
 \[
 \ \int \dfrac{f_0(z)  - \pi i \frac{d\eta_0}{dl_{\partial_i K}}- f_0(\lambda)}{z - \lambda} g_j(z)d\mu (z) = 0.
 \]
  Therefore,  $\rho(f_0 - \pi i \frac{d\eta_0}{dl_{\partial_i K}}) = f_0|_\Omega \in H^\i (\Omega, E_0).$
 \[
 \ f = \dfrac{f_0 - \pi i \frac{d\eta_0}{dl_{\partial_i K}}}{\|\eta_0\|} \gamma (E_1) \text{ and } \eta  = \dfrac{\eta_0}{\|\eta_0\|} \gamma (E_1). 
 \]
 
 Case III: Similarly  to Case II, we set 
 \[
 \ f = \dfrac{f_0 + \pi i \frac{d\eta_0}{dl_{\partial_i K}}}{\|\eta_0\|} \gamma (E_1) \text{ and } \eta  = \dfrac{\eta_0}{\|\eta_0\|} \gamma (E_1). 
 \] 
 It is easy to verify that $f$ and $\eta$ satisfy the properties of the lemma.     
\end{proof}

Define 
\begin{eqnarray}\label{NRBDef}
\ \mathcal F_1 = F(\{g_j\mu\}) \cup \T \cup \bigcup_{n=1}^\infty \partial \D(\lambda_n, r_n).
\end{eqnarray}
Clearly, $\mathcal F_1 = \partial\Omega \setminus E_0,~ \gamma-a.a..$The set $\mathcal F := \ \mathcal F_1 \cup K^c$ is called the non-removable boundary for $\rtkmu.$

In the remaining section, we prove the following integral estimate. 

\begin{lemma}\label{IEstSOB}
Let $S$ be a closed square whose sides are parallel to $x$-axis or $y$-axis and let the length of $S$ is $\delta.$
Let $\varphi$ be a smooth function with support in $S,$ $\|\varphi\|_\i \le 1,$ and $\left \|\frac{\partial \varphi (z)}{\partial  \bar z} \right \|_\infty \le \frac{C_4}{\delta}.$ Then for $f\in H^\infty (\Omega, E_0)$ with $\|f\|_\Omega \le 1$ and $f(z) = 0, ~ z\in \C \setminus K,$ and $n \ge 1,$
 \begin{eqnarray}\label{IEstSOBEq}
 \begin{aligned}
 \ &\left | \int f(z) (z - \lambda)^{n-1}\dfrac{\partial \varphi (z)}{\partial  \bar z} d\area(z) \right | \\
 \ \le &C_4\delta^{n-1} l_{\partial \Omega} (S \cap \mathcal F_1).
 \end{aligned}
 \end{eqnarray}	
\end{lemma}

\begin{proof}
For $f\in H^\infty (\Omega, E_0)$ with $f(z) = 0,~ z\in \C \setminus K$ and $\|f\|_\Omega \le 1,$ applying Lemma \ref{BAForH1Finite} for $\Gamma=\partial \Omega,$ we find a function $b(z) \in L^\i(l_{\partial \Omega})$ with $\|b\| \le 1$ such that $f(z) = \CT(bl_{\partial \Omega})(z)$ for $z\in \C\setminus \partial \Omega.$ From Lemma \ref{GPTheorem}, we have
\[
\ f_+(z) = v^+(bl_{\partial \Omega})(z) = \CT(bl_{\partial \Omega})(z) + \pi i b(z),~ l_{\partial _i K}-a.a.
\]
and
\[
\ f_-(z) = v^-(bl_{\partial \Omega})(z) = \CT(bl_{\partial \Omega})(z) - \pi i b(z),~ l_{\partial _i K}-a.a..
\]
Therefore, $b(z) = 0,~ l_{E_0}-a.a.$ since $f_+(z) = f_-(z),~ l_{E_0}-a.a..$ For a smooth function $\varphi$ with support in $S,$ $\|\varphi\|_\i \le 1,$ and $\left \|\frac{\partial \varphi (z)}{\partial  \bar z} \right \|_\infty \le \frac{C_4}{\delta}.$ we get
 \[
 \begin{aligned}
 \ &\left | \int f(z) (z - \lambda)^{n-1} \dfrac{\partial \varphi (z)}{\partial  \bar z} d\area(z) \right | \\
 \ \le & \pi \left | \int_{\partial \Omega \setminus E_0} (z - \lambda)^{n-1} b(z) \varphi (z) dl_{\partial \Omega} (z) \right | \\
 \ \le & C_5\delta^{n-1} l_{\partial \Omega} (S \cap (\partial \Omega \setminus E_0)).
 \end{aligned}
 \]
 This completes the proof.
\end{proof}

\section{\textbf{Proofs of Main Theorem (2) and (3)}}

We continue to use the assumptions and notation of sections 5 and 6. 
If  a compact subset $E$ is contained in the disk $\D(a, \delta)$ and $f$ is bounded analytic on $\C_\i \setminus E$ with $f(\infty) = 0.$ We consider the Laurent expansion of $f$ centered at $a$ for $z\in \mathbb C\setminus \D(a, \delta),$

\begin{eqnarray}\label{LaurentExpansion}
 \ f(z) = \sum_{n=1}^\infty \dfrac{c_n(f,a)}{(z-a)^n}.
 \end{eqnarray}
 As $c_1(f, a)$ does not depend on the choice of $a$,  we define 
 $c_1 (f) = c_1(f, a).$ The coefficient $c_2 (f,a) = c_2(f, a)$ does  depend on $a$.
 
 Let $\varphi$ be a smooth function with compact support. Vitushkin's localization operator
$T_\varphi$ is defined by
 \[
 \ (T_\varphi f)(\lambda) = \dfrac{1}{\pi}\int \dfrac{f(z) - f(\lambda)}{z - \lambda} \bar\partial \varphi (z) d\area(z),
 \]
where $f\in L^1_{loc} (\C)$.	 Clearly, $(T_\varphi f)(z) = - \frac{1}{\pi}\mathcal C(\varphi \bar\partial f\area ) (z).$
Consequently, in the sense of distributions,
\[
\ \bar \partial (T_\varphi f)(z) = \varphi (z) \bar \partial f(z). 
\]
Therefore, $T_\varphi f$ is analytic outside of $\text{supp} (\bar \partial f) \cap\text{supp} (\varphi).$ If $\text{supp} (\varphi) \subset \D(a,\delta),$ then  
\[
\begin{aligned}
  \ & (T_\varphi f)(\i ) = 0,~ c_1(T_\varphi f) = - \dfrac{1}{\pi}\int f(z)\bar\partial \varphi (z) d\area(z), \\
  \ & c_2(T_\varphi f,a) =  \dfrac{1}{\pi}\int (z-a)f(z)\bar\partial \varphi (z) d\area(z).
\end{aligned}
\]
The following is the estimate of the norm of $T_\varphi f.$
\[
 \ \| T_\varphi f\|_\i   \le  4\|f\|_\i  \delta\|\bar\partial \varphi\|.
 \]	
 See \cite[VIII.7.1]{gamelin} for the details if $T_\varphi.$

\begin{lemma}\label{CaseAB} 
Let $S$ be a closed square whose sides are parallel to $x$-axis or $y$-axis such that $S\subset \D$ and $S\cap \R \ne \emptyset$ and let $\delta$ and $c_S$ be the length and the center of $S,$ respectively.  
Let $\varphi$ be a smooth function with support in $S,$ $\|\varphi\|_\i \le 1,$ and $\left \|\frac{\partial \varphi (z)}{\partial  \bar z} \right \|_\infty \le \frac{C_6}{\delta}.$ Let $f\in H^\infty (\Omega, E_0)$ with $\|f\|_\Omega \le 1$ and $f(z) = 0, ~ z\in \C \setminus K.$
Then there exists $f_S\in R^{t,\i}(K,\mu)$ that is bounded analytic on $\C_\i\setminus S$ satisfying the following properties hold

(1) $\|f_S\|_{\Omega} \le C_7,$  

(2) $\|f_S\|_{L^\infty(\mu)} \le C_7$ and $f_S(\infty) = 0,$ 

(3) $c_1(f_S) = c_1(T_\varphi f),$ and 

(4) $\rho(f_S) = f_S|_\Omega \in H^\i (\Omega, E_0).$
\end{lemma}

\begin{proof}
Let 
\[
\ D_S = \bigcup_{\overline{\D(\lambda_n, r_n)} \subset S} \D(\lambda_n, r_n),~ T_S = \bigcup_{\overline{\D(\lambda_n, r_n)} \subset S} \partial {\D(\lambda_n, r_n)}.
\]
There are at most two disks $D_l = \D(\lambda_{n_l}, r_{n_l})$ and $D_r = \D(\lambda_{n_r}, r_{n_r})$ with $\lambda_{n_l} < \lambda_{n_r}$ such that $\overline {D_l} \cap (S \setminus D_S) \ne \emptyset$ and $\overline {D_r} \cap (S \setminus D_S)\ne \emptyset.$ 
Then we have
\[
\begin{aligned}
\ &l_{\partial \Omega}(S\cap \mathcal F_1) \le l_{\partial \Omega}(S \cap F(\{g_j\mu\})) \\
\ & + l_{\partial \Omega}(S \cap T_S) + l_{\partial \Omega}(S \cap \partial  D_l) + l_{\partial \Omega}(S \cap \partial  D_r).
\end{aligned}
\]

Case A: Assume that $l_{\partial \Omega}(S \cap F(\{g_j\mu\})) \ge \frac{1}{4}l_{\partial \Omega}(S\cap \mathcal F_1).$ Let $f_0$ be the function in Lemma \ref{BBFunctions} for $E_1 = S \cap F(\{g_j\mu\}).$ Set $f_S = \frac{f_0}{c_1(f_0)}c_1(T_\varphi f).$ (1)-(4) follows from Lemma \ref{BBFunctions} and Lemma \ref{IEstSOB}.

Case B: Assume that $l_{\partial \Omega}(S \cap F(\{g_j\mu\})) < \frac{1}{4}l_{\partial \Omega}(S\cap \mathcal F_1).$ Let $E_2$ be one of $S \cap T_S,$ $S \cap \partial D_l,$ and $S \cap \partial D_r$ with the largest $l_{\partial \Omega}$ measure. Then $l_{\partial \Omega}(E_2) \ge \frac{1}{4}l_{\partial \Omega}(S\cap \mathcal F_1).$ We can find a compact subset of a line segment 
\[
\ E_3\subset S\cap(D_S\cup D_l\cup D_r)
\]
such that $l_{\partial \Omega}(E_3) \ge \frac 12 l_{\partial \Omega}(E_2).$
Let $f_0$ be the admissible function for $E_3,$ that is, $f_0$ is analytic on $\C_\i \setminus E_3,$ $\|f_0\| = 1,$ and $f_0'(\infty) = \gamma(E_3).$ Using \cite[VIII.2.2]{gamelin} and Lemma \ref{IEstSOB}, we conclude that the function $f_S = \frac{f_0}{f_0'(\infty)}c_1(T_\varphi f)$ satisfies (1)-(4).
\end{proof}
 
 \begin{lemma}\label{CaseC} 
 Let $S$ be a closed square whose sides are parallel to $x$-axis or $y$-axis and let $\delta$ and $c_S$ be the length and the center of $S,$ respectively.  
Suppose that (a) $S \setminus \D \ne \emptyset$ or (b) $dist(c_S, \R) \ge \delta$ and $S \cap \partial K \ne \emptyset.$ Let $\varphi$ be a smooth function with support in $S,$ $\|\varphi\|_\i \le 1,$ and $\left \|\frac{\partial \varphi (z)}{\partial  \bar z} \right \|_\infty \le \frac{C_8}{\delta}.$ Let $f\in H^\infty (\Omega, E_0)$ with $\|f\|_\Omega \le 1$ and $f(z) = 0, ~ z\in \C \setminus K.$
Then there exists $f_S\in R^{t,\i}(K,\mu)$ that is bounded analytic on $\C_\i\setminus (\frac 32 S)$ satisfying the following properties hold

(1) $\|f_S\|_{\Omega} \le C_9,$  

(2) $\|f_S\|_{L^\infty(\mu)} \le C_9$ and $f_S(\infty) = 0,$ 

(3) $c_1(f_S) = c_1(T_\varphi f),$

(4) $c_2(f_S,c_S) = c_2(T_\varphi f,c_S),$  and  

(5) $\rho(f_S) = f_S|_\Omega \in H^\i (\Omega, E_0).$
\end{lemma}

\begin{proof}
If (a) holds, then there exists a line segment $L$ with length greater than $\frac 14 \delta$ such that $L \subset \frac 32 S \setminus \overline \D.$ If $S \subset \D$ and (b) holds, then there exists $\D(\lambda_{n_0},	r_{n_0})$ such that $S \cap \partial \D(\lambda_{n_0},	r_{n_0}) \ne \emptyset.$ Hence, $r_{n_0} \ge \frac 12 \delta$ and there exists a line segment $L$ with length greater than $\frac 14 \delta$ such that $L \subset \frac 32 S \cap\D(\lambda_{n_0},	r_{n_0}).$

There exists a function $f_1\in C(\C_\i)$ such that $\|f_1\| \le 1,$
$f_1$ is analytic on $\C_\i\setminus L,$ $f_1(\infty) = 0,$ and $c_1(f_1) = \frac {\delta}{16}.$ Set $f_2 = \frac {\delta f_1}{16c_1(f_1)}.$ Then $\|f_2\| \le 1,$  $c_1(f_2^2) = 0,$ and $c_2 (f_2^2,c_S) = \frac{\delta^2}{256}.$ By \cite[Theorem 2.5 on page 201]{gamelin}, we get $|c_2 (f_2,c_S)|  \le 3 \delta^2.$ Let
 \[
 \ f_3 = f_2 - \dfrac{256c_2 (f_2,c_S)}{\delta^2} f_2^2.
 \]
 Then $\|f_3\| \le C_{10},$ $c_1(f_3) = \frac{\delta}{16},$ and $c_2 (f_3,c_S) = 0,$ Set  
 \[
 \ f_S =  \dfrac{16c_1(T_\varphi f)}{\delta} f_3 + \dfrac{256c_2(T_\varphi f,c_S)}{\delta^2} f_2^2.
 \]
 The proof now follows from  Lemma \ref{IEstSOB}.
\end{proof}

Let $\delta > 0$. We say that 
$\{\varphi_{ij},S_{ij}, \delta\}$ is a smooth partition of unity subordinated to $\{2S_{ij}\},$ 
 if the following assumptions hold:

(i)  $S_{ij}$ is a square with vertices   $(i\delta,j\delta),~((i+1)\delta,j\delta),~(i\delta,(j+1)\delta),$ and $((i+1)\delta,(j+1)\delta);$

(ii) $\varphi_{ij}$ is a $C^\infty$ smooth function supported in $2S_{ij},$
and with values in $[0,1]$;

(iii) 
 \[
 \ \|\bar\partial \varphi_{ij} \| \le \frac{C_{1}}{\delta},~ \sum \varphi_{ij} = 1.
 \]
Given a square grid $\{ S_{ij} \}$, we shall let $c_{ij}$ denote the center of $S_{ij}$ (See \cite[VIII.7]{gamelin} for details).

\begin{proof} (Main Theorem (2))
Let $\{\varphi_{ij},S_{ij}, \delta\}$ be a smooth partition of unity. 
Let $f\in H^\infty(\Omega, E_0),$ $\|f\|\le 1,$ and set $f(z) = 0$ for $z\in \mathbb C \setminus \Omega.$ Then 
\begin{eqnarray}\label{ImRhoOntoEq1}
 \ f = \sum_{2S_{ij}\cap \partial \Omega \ne \emptyset} (f_{ij}:= T_{\varphi_{ij}}f). 
 \end{eqnarray}
 Let $J$ be the set of $(i,j)$ satisfying $i\in \{-2,-1,0,1\}$ and $2S_{ij}\subset \D.$ 
 We apply Lemma \ref{CaseAB} when $(i,j) \in J$ and Lemma \ref{CaseC} when $(i,j) \notin J$ for $S= 2S_{ij}$ and $\varphi = \varphi_{ij}$ to get $f_{S_{ij}}.$ 
We rewrite \eqref{ImRhoOntoEq1} as the following
 \[
 \ f = \sum_{(i,j) \in J,~ 2S_{ij}\cap \partial \Omega \ne \emptyset} (f_{ij} - f_{S_{ij}}) + \sum_{(i,j) \notin J,~ 2S_{ij}\cap \partial \Omega \ne \emptyset} (f_{ij} - f_{S_{ij}}) + f_{\delta}
 \]
where
 \begin{eqnarray}\label{FDeltaSOB}
 \ f_{\delta} = \sum_{2S_{ij}\cap \partial \Omega \ne \emptyset}f_{S_{ij}} \in R^{t,\i}(K,\mu).
 \end{eqnarray}
 Clearly, $\rho(f_{\delta}) = f_{\delta}|_\Omega \in H^\i (\Omega, E_0).$ Using Vitushkin approximation scheme (see \cite[VIII.7]{gamelin}), we get
 \[
 \begin{aligned}
 \ & |f(z) - f_\delta(z)| \\
 \ \le & \sum_{(i,j) \in J,~ 2S_{ij}\cap \partial \Omega \ne \emptyset} |f_{ij}(z) - f_{S_{ij}}(z)| + \sum_{(i,j) \notin J,~ 2S_{ij}\cap \partial \Omega \ne \emptyset} |f_{ij}(z) - f_{S_{ij}}(z)| \\
 \ \le & C_{11} \sum_{\underset{-2\le i \le 1}{2S_{ij}\cap \partial \Omega \ne \emptyset}} \min \left (1, \dfrac{\delta ^2}{dist(z, 3S_{ij})^2} \right ) + C_9 \sum_{2S_{ij}\cap \partial \Omega \ne \emptyset} \min \left (1, \dfrac{\delta ^3}{dist(z, 3S_{ij})^3} \right ) \\
 \ \le & C_{12} \min \left (1, \dfrac{\delta}{dist(z, \partial \Omega)} \right ).
 \end{aligned}
 \]  
Therefore, $f_{\delta} (z) \rightarrow f(z)$ uniformly on compact subsets of $\mathbb C \setminus \partial \Omega$ and $\|f_{\delta}\|_{L^\infty(\mu)} \le C_{13}.$ There exists a subsequence $\{f_{\delta_n}\}$ such that $f_{\delta_n} (z) \rightarrow \tilde f(z)$ in $L^\infty(\mu)$ weak-star topology. Hence, $\tilde f\in R^{t,\i}(K,\mu)$ and $\mathcal C(f_{\delta_n} g\mu)(z) \rightarrow \mathcal C(\tilde fg\mu)(z), ~\area-a.a..$ Thus, $f(z)\mathcal C(g\mu)(z) = \mathcal C(\tilde fg\mu)(z), ~\area-a.a.$ for $g\perp R^t(K,\mu),$ which implies $\rho(\tilde f) = f.$ 
\end{proof}

We are now ready to prove Main Theorem (3).

\begin{proof} (Main Theorem (3))
From Main Theorem (1), we see that $\rho(f)\in H^\infty (\Omega, E_0)$ for $f\in R^{t,\i}(K,\mu).$	 Clearly, 
\[
\ \rho(f_1f_2) = \rho(f_1)\rho(f_2),~f_1,f_2\in R^{t,\i}(K,\mu). 
\]
Therefore, for $f\in  R^{t,\i}(K,\mu),$ since $\lambda\in \Omega$ is an analytic bounded point evaluation, we see that there exists a constant $C_\lambda > 0$ such that
\[
\ |\rho(f)(\lambda)| \le C_\lambda^{\frac 1n} \|f^n\|_{L^t(\mu)}^{\frac 1n}. 
\]
Taking $n\rightarrow \infty,$ we get $\|\rho(f)\|_\Omega\le \|f\|_{L^\infty(\mu)}.$ 

From the last paragraph of the proof of Main Theorem (2), we get $\|f^n\|_{L^\infty(\mu)} ^{\frac 1n}\le C_{13} ^{\frac 1n} \|\rho(f^n)\|_\Omega^{\frac 1n}.$ 
Again letting 
$n\rightarrow \infty,$ we get $\|f\|_{L^\infty(\mu)} \le \|\rho(f)\|_\Omega.$ Thus,
\[
\ \|\rho(f)\|_\Omega = \|f\|_{L^\infty(\mu)},~ f\in \in R^{t,\i}(K,\mu).
\]
On the other hand, $\rho$ is surjective by Main Theorem (2). Therefore, $\rho$ is a bijective isomorphism between two Banach algebras $R^{t,\i}(K, \mu)$ and $H^\infty (\Omega,E_0)$. Clearly $\rho$ is also weak-star sequentially continuous, so an application of the Krein-Smulian Theorem shows that $\rho$ is a weak-star homeomorphism.
\end{proof}
\bigskip

{\bf Acknowledgments.} 
The authors would like to thank Professor John M\raise.45ex\hbox{c}Carthy for carefully reading through the manuscript and providing many useful comments.
They would also like to thank Professor John Akeroyd for providing useful discussions.
\bigskip

\bibliographystyle{amsplain}

\end{document}